\documentclass[reqno]{amsart}
\usepackage{amsmath}
\usepackage{amsfonts}

\newtheorem{theorem}{Theorem}
\theoremstyle{plain}

\newtheorem{definition}{Definition}
\newtheorem{example}{Example}

\newtheorem{proposition}{Proposition}

\numberwithin{equation}{section}

\begin{document}
\title[$C$-parallel and $C$-proper Curves]{$C$-parallel and $C$-proper Slant
Curves of $S$-manifolds}
\author{\c{S}aban G\"{u}ven\c{c}}
\address[\c{S}. G\"{u}ven\c{c} and C. \"{O}zg\"{u}r]{Balikesir University,
Department of Mathematics\\
Campus of Cagis, Balikesir, TURKEY}
\email[\c{S}. G\"{u}ven\c{c}]{sguvenc@balikesir.edu.tr}
\author{Cihan \"{O}ZG\"{U}R}
\email[C.~\"{O}zg\"{u}r]{cozgur@balikesir.edu.tr}

\begin{abstract}
In the present paper, we define and study $C$-parallel and $C$-proper slant
curves of $S$-manifolds. We prove that a curve $\gamma $ in an $S$-manifold
of order $r\geq 3,$ under certain conditions, is $C$-parallel or $C$%
-parallel in the normal bundle if and only if it is a non-Legendre slant
helix or Legendre helix, respectively. Moreover, under certain conditions,
we show that $\gamma $ is $C$-proper or $C$-proper in the normal bundle if
and only if it is a non-Legendre slant curve or Legendre curve,
respectively. We also give two examples of such curves in $\mathbb{R}%
^{2m+s}(-3s).$
\end{abstract}

\maketitle

\section{Introduction\label{sect-introduction}}

Let $M^{m}$ be an integral submanifold of a Sasakian manifold $%
(N^{2n+1},\varphi ,\xi ,\eta ,g)$. Then $M$ is called \textit{integral }$C$%
\textit{-parallel} if $\nabla ^{\perp }B$ is parallel to the characteristic
vector field $\xi $, where $B$ is the second fundamental form of $M$ and $%
\nabla ^{\perp }B$ is given by%
\begin{equation*}
(\nabla ^{\perp }B)(X,Y,Z)=\nabla _{X}^{\perp }B(Y,Z)-B(\nabla
_{X}Y,Z)-B(Y,\nabla _{X}Z)
\end{equation*}%
where $X,Y,Z$ are vector fields on $M$, $\nabla ^{\perp }$ and $\nabla $ are
the normal connection and the Levi-Civita connection on $M$, respectively
\cite{FO-2012}. Now, let $\gamma $ be a curve in an almost contact metric
manifold $(M,\varphi ,\xi ,\eta ,g)$. Lee, Suh and Lee introduced the
notions of $C$-parallel and $C$-proper curves along slant curves of Sasakian
$3$-manifolds in the tangent and normal bundles \cite{LSL}. A curve $\gamma $
in an almost contact metric manifold $(M,\varphi ,\xi ,\eta ,g)$ is said to
be $C$\textit{-parallel }if $\nabla _{T}H=\lambda \xi $, $C$\textit{-proper }%
if $\Delta H=\lambda \xi $, $C$\textit{-parallel in the normal bundle }if $%
\nabla _{T}^{\perp }H=\lambda \xi $, $C$\textit{-proper in the normal bundle}
if $\Delta ^{\perp }H=\lambda \xi $, where $T$ is the unit tangent vector
field of $\gamma $, $H$ is the mean curvature vector field, $\Delta $ is the
Laplacian, $\lambda $ is a non-zero differentiable function along the curve $%
\gamma ,$ $\nabla ^{\perp }$ and $\Delta ^{\perp }$ denote the normal
connection and Laplacian in the normal bundle, respectively \cite{LSL}. For
a submanifold $M$ of an arbitrary Riemannian manifold $\widetilde{M}$, if $%
\Delta H=\lambda H$, then $M$ is called \textit{submanifold with a proper
mean curvature vector field }$H$ \cite{Chen-89}. If $\Delta ^{\perp
}H=\lambda H$, then $M$ is said to be \textit{submanifold with a proper mean
curvature vector field} $H$ \textit{in the normal bundle} \cite{Arroyo}.

Let $\gamma (s)$ be a Frenet curve parametrized by the arc-length parameter $%
s$ in an almost contact metric manifold $M$. The function $\theta (s)$
defined by $cos[\theta (s)]=g(T(s),\xi )$ is called \textit{the contact
angle function}. A curve $\gamma $ is called a \textit{slant curve} if its
contact angle is a constant \cite{CIL}. If a slant curve is with contact
angle $\frac{\pi }{2}$, then it is called a \textit{Legendre curve }\cite%
{Blair}.

Lee, Suh and Lee studied $C$-parallel and $C$-proper slant curves of
Sasakian $3$-manifolds in \cite{LSL}. As a generalization of this paper, in
\cite{GO-revuma}, the present authors studied $C$-parallel and $C$-proper
curves in trans-Sasakian manifolds. In the present paper, our aim is to
consider $C$-parallel and $C$-proper curves of $S$-manifolds.

The paper is organized as follows: In Section \ref{sect-preliminaries}, we
give a brief introduction about $S$-manifolds. Futhermore, we define the
notions of $C$-parallel and $C$-proper curves in $S$-manifolds both in
tangent and normal bundles. In Section \ref{sect-cparallel}, we consider $C$%
-parallel slant curves in $S$-manifolds in tangent and normal bundles,
respectively. In Section \ref{sect-cproper}, we study $C$-proper slant
curves in $S$-manifolds in tangent and normal bundles, respectively. In the
final section, we present two examples of these kinds of curves in $\mathbb{R%
}^{2m+s}(-3s)$.

\section{\textbf{Preliminaries}\label{sect-preliminaries}}

Let $(M,g)$ be a $(2m+s)$-dimensional Riemann manifold. $M$ is called
\textit{framed metric manifold} \cite{YK-1984} with a \textit{framed metric
structure} $(\varphi ,\xi _{\alpha },\eta ^{\alpha },g)$, $\alpha \in
\left\{ 1,...,s\right\} ,$ if this structure satisfies the following
equations:
\begin{equation}
\begin{array}{cccc}
\varphi ^{2}=-I+\overset{s}{\underset{\alpha =1}{\sum }}\eta ^{\alpha
}\otimes \xi _{\alpha }, & \eta ^{\alpha }(\xi _{\beta })=\delta _{\beta
}^{\alpha }, & \varphi \left( \xi _{\alpha }\right) =0, & \eta ^{\alpha
}\circ \varphi =0%
\end{array}%
\end{equation}%
\begin{equation}
g(\varphi X,\varphi Y)=g(X,Y)-\overset{s}{\underset{\alpha =1}{\sum }}\eta
^{\alpha }(X)\eta ^{\alpha }(Y),
\end{equation}%
\begin{equation}
\begin{array}{cc}
d\eta ^{\alpha }(X,Y)=g(X,\varphi Y)=-d\eta ^{\alpha }(Y,X), & \eta ^{\alpha
}(X)=g(X,\xi ),%
\end{array}%
\end{equation}%
where, $\varphi $ is a ($1,1$) tensor field\textit{\ }of rank $2m$; $\xi
_{1},...,\xi _{s}$ are vector fields; $\eta ^{1},...,\eta ^{s}$ are $1$%
-forms and $g$ is a Riemannian metric on $M$; $X,Y\in TM$ and $\alpha ,\beta
\in \left\{ 1,...,s\right\} $. $(M^{2m+s},\varphi ,\xi _{\alpha },\eta
^{\alpha },g)$ is also called \textit{framed }$\varphi $\textit{-manifold }%
\cite{Nak-1966} or \textit{almost }$r$\textit{-contact metric manifold }\cite%
{Vanzura-1972}. $(\varphi ,\xi _{\alpha },\eta ^{\alpha },g)$ is said to be
an $S$\textit{-structure},\textit{\ }if the Nijenhuis tensor of $\varphi $
is equal to $-2d\eta ^{\alpha }\otimes \xi _{\alpha }$, where $\alpha \in
\left\{ 1,...,s\right\} $ \cite{Blair-1970, Cabrerizo}.

When $s=1$, a framed metric structure turns into an almost contact metric
structure and an $S$-structure turns into a Sasakian structure. For an $S$%
-structure, the following equations are satisfied \cite{Blair-1970, Cabrerizo}:
\begin{equation}
(\nabla _{X}\varphi )Y=\underset{\alpha =1}{\overset{s}{\sum }}\left\{
g(\varphi X,\varphi Y)\xi _{\alpha }+\eta ^{\alpha }(Y)\varphi ^{2}X\right\}
,  \label{nablaf}
\end{equation}%
\begin{equation}
\nabla _{X}\xi _{\alpha }=-\varphi X,\text{ }\alpha \in \left\{
1,...,s\right\} .  \label{nablaxi}
\end{equation}%
If $M$ is Sasakian ($s=1$), (\ref{nablaxi}) can be directly\ calculated from
(\ref{nablaf}).

Firstly, we give the following definition:

\begin{definition}
Let $\gamma :I\rightarrow (M^{2m+s},\varphi ,\xi _{\alpha },\eta ^{\alpha
},g)$ be a unit speed curve in an $S$-manifold. Then $\gamma $ is called

i) $C$\textit{-parallel (in the tangent bundle) }if
\begin{equation*}
\nabla _{T}H=\lambda \overset{s}{\underset{\alpha =1}{\sum }}\xi _{\alpha },
\end{equation*}

ii) $C$\textit{-parallel in the normal bundle }if
\begin{equation*}
\nabla _{T}^{\perp }H=\lambda \overset{s}{\underset{\alpha =1}{\sum }}\xi
_{\alpha },
\end{equation*}

iii) $C$\textit{-proper (in the tangent bundle) }if
\begin{equation*}
\Delta H=\lambda \overset{s}{\underset{\alpha =1}{\sum }}\xi _{\alpha },
\end{equation*}

iv) $C$\textit{-proper in the normal bundle} if
\begin{equation*}
\Delta ^{\perp }H=\lambda \overset{s}{\underset{\alpha =1}{\sum }}\xi
_{\alpha },
\end{equation*}%
where $H$ is the mean curvature field of $\gamma $, $\lambda $ is a
real-valued non-zero differentiable function and $\Delta $ is the Laplacian.
\end{definition}

Let $\gamma :I\rightarrow M$ \ be a curve parametrized by arc length in an $%
n $-dimensional Riemannian manifold $(M,g)$. Denote by the Frenet frame and
curvatures of $\gamma $ by $\left\{ E_{1},E_{2},...,E_{r}\right\} $ and $%
\kappa _{1},...,\kappa _{r-1},$ respectively. We know that (see \cite{Arroyo}%
)%
\begin{equation*}
\nabla _{T}H=-\kappa _{1}^{2}E_{1}+\kappa _{1}^{\prime }E_{2}+\kappa
_{1}\kappa _{2}E_{3},
\end{equation*}%
\begin{equation*}
\nabla _{T}^{\perp }H=\kappa _{1}^{\prime }E_{2}+\kappa _{1}\kappa _{2}E_{3},
\end{equation*}%
\begin{eqnarray*}
\Delta H &=&-\nabla _{T}\nabla _{T}\nabla _{T}T \\
&=&3\kappa _{1}\kappa _{1}^{\prime }E_{1}+\left( \kappa _{1}^{3}+\kappa
_{1}\kappa _{2}^{2}-\kappa _{1}^{\prime \prime }\right) E_{2} \\
&&-(2\kappa _{1}^{\prime }\kappa _{2}+\kappa _{1}\kappa _{2}^{\prime
})E_{3}-\kappa _{1}\kappa _{2}\kappa _{3}E_{4}
\end{eqnarray*}%
and%
\begin{eqnarray*}
\Delta ^{\perp }H &=&-\nabla _{T}^{\perp }\nabla _{T}^{\perp }\nabla
_{T}^{\perp }T \\
&=&\left( \kappa _{1}\kappa _{2}^{2}-\kappa _{1}^{\prime \prime }\right)
E_{2}-\left( 2\kappa _{1}^{\prime }\kappa _{2}+\kappa _{1}\kappa
_{2}^{\prime }\right) E_{3} \\
&&-\kappa _{1}\kappa _{2}\kappa _{3}E_{4}.
\end{eqnarray*}%
So we can directly state the following Proposition:

\begin{proposition}
\label{lemma1}Let $\gamma :I\rightarrow (M^{2m+s},\varphi ,\xi _{\alpha
},\eta ^{\alpha },g)$ be a unit speed curve in an $S$-manifold. Then

i) $\gamma $ is $C$-parallel (in the tangent bundle) if and only if
\begin{equation}
-\kappa _{1}^{2}E_{1}+\kappa _{1}^{\prime }E_{2}+\kappa _{1}\kappa
_{2}E_{3}=\lambda \overset{s}{\underset{\alpha =1}{\sum }}\xi _{\alpha },
\label{cparalleltangent}
\end{equation}

ii) $\gamma $ is $C$\textit{-parallel in the normal bundle }if and only if%
\begin{equation}
\kappa _{1}^{\prime }E_{2}+\kappa _{1}\kappa _{2}E_{3}=\lambda \overset{s}{%
\underset{\alpha =1}{\sum }}\xi _{\alpha },  \label{cparallelnormal}
\end{equation}

iii) $\gamma $ is $C$\textit{-proper (in the tangent bundle) }if and only if
\begin{equation}
3\kappa _{1}\kappa _{1}^{\prime }E_{1}+\left( \kappa _{1}^{3}+\kappa
_{1}\kappa _{2}^{2}-\kappa _{1}^{\prime \prime }\right) E_{2}-(2\kappa
_{1}^{\prime }\kappa _{2}+\kappa _{1}\kappa _{2}^{\prime })E_{3}-\kappa
_{1}\kappa _{2}\kappa _{3}E_{4}=\lambda \overset{s}{\underset{\alpha =1}{%
\sum }}\xi _{\alpha },  \label{cpropertangent}
\end{equation}

iv) $\gamma $ is $C$\textit{-proper in the normal bundle} if and only if
\begin{equation}
\left( \kappa _{1}\kappa _{2}^{2}-\kappa _{1}^{\prime \prime }\right)
E_{2}-\left( 2\kappa _{1}^{\prime }\kappa _{2}+\kappa _{1}\kappa
_{2}^{\prime }\right) E_{3}-\kappa _{1}\kappa _{2}\kappa _{3}E_{4}=\lambda
\overset{s}{\underset{\alpha =1}{\sum }}\xi _{\alpha }.
\label{cpropernormal}
\end{equation}
\end{proposition}

Now, our aim is to apply Proposition \ref{lemma1} to slant curves in $S$%
-manifolds.

Let $\gamma :I\rightarrow (M^{2m+s},\varphi ,\xi _{\alpha },\eta ^{\alpha
},g)$ be a slant curve. Then, if we differentiate%
\begin{equation*}
\eta ^{\alpha }(T)=\cos \theta ,
\end{equation*}%
we get%
\begin{equation*}
\eta ^{\alpha }(E_{2})=0,
\end{equation*}
where, $\theta $ denotes the constant contact angle satisfying%
\begin{equation*}
\frac{-1}{\sqrt{s}}\leq \cos \theta \leq \frac{1}{\sqrt{s}}.
\end{equation*}%
The equality case is only valid for geodesics corresponding to the integral
curves of
\begin{equation*}
T=\frac{\pm 1}{\sqrt{s}}\overset{s}{\underset{\alpha =1}{\sum }}\xi _{\alpha
},
\end{equation*}%
(see \cite{GO2018}).

\section{$C$-parallel Slant Curves of $S$-manifolds\label{sect-cparallel}}

\bigskip Our first Theorem below is a result of Proposition \ref{lemma1} i).

\begin{theorem}
Let $\gamma :I\rightarrow M^{2m+s}$ be a unit-speed slant curve. Then $%
\gamma $ is $C$-parallel (in the tangent bundle) if and only if it is a
non-Legendre slant helix of order $r\geq 3$ satisfying%
\begin{equation*}
\overset{s}{\underset{\alpha =1}{\sum }}\xi _{\alpha }\in sp\left\{
T,E_{3}\right\} ,
\end{equation*}%
\begin{equation*}
\varphi T\in sp\left\{ E_{2},E_{4}\right\} ,
\end{equation*}%
\begin{equation*}
\kappa _{2}=\frac{-\kappa _{1}\sqrt{1-s\cos ^{2}\theta }}{\sqrt{s}\cos
\theta },\text{ }\kappa _{2}\neq 0,
\end{equation*}%
\begin{equation*}
\lambda =\frac{-\kappa _{1}^{2}}{s\cos \theta }=\text{constant,}
\end{equation*}%
and moreover if $\kappa _{3}=0$, then%
\begin{equation}
\kappa _{1}=-s\cos \theta \sqrt{1-s\cos ^{2}\theta },  \label{eq3}
\end{equation}%
\begin{equation}
\kappa _{2}=\sqrt{s}\left( 1-s\cos ^{2}\theta \right) .  \label{eq4}
\end{equation}
\end{theorem}

\begin{proof}
Let us assume that $\gamma $ is $C$-parallel (in the tangent bundle). Then,
if we apply $E_{2}$ to equation (\ref{cparalleltangent}), we find $\kappa
_{1}^{\prime }=0$, that is, $\kappa _{1}=$constant. Now, applying $T$ to (%
\ref{cparalleltangent}), we have%
\begin{equation*}
\lambda s\cos \theta =-\kappa _{1}^{2}.
\end{equation*}%
Here, $\theta \neq \frac{\pi }{2}$ since $\kappa _{1}\neq 0$. Hence, $\gamma
$ is non-Legendre slant. So, we get%
\begin{equation*}
\lambda =\frac{-\kappa _{1}^{2}}{s\cos \theta }=\text{constant.}
\end{equation*}%
Equation (\ref{cparalleltangent}) can be rewritten as%
\begin{equation*}
\overset{s}{\underset{\alpha =1}{\sum }}\xi _{\alpha }=\frac{-\kappa _{1}^{2}%
}{\lambda }T+\frac{\kappa _{1}\kappa _{2}}{\lambda }E_{3},
\end{equation*}%
which is equivalent to%
\begin{equation}
\overset{s}{\underset{\alpha =1}{\sum }}\xi _{\alpha }=s\cos \theta T-\frac{%
\kappa _{2}s\cos \theta }{\kappa _{1}}E_{3}.  \label{eq1}
\end{equation}%
If we calculate the norm of both sides, we obtain%
\begin{equation}
\kappa _{2}=\frac{-\kappa _{1}\sqrt{1-s\cos ^{2}\theta }}{\sqrt{s}\cos
\theta }.  \label{eq2}
\end{equation}%
If we assume $\kappa _{2}=0$, we have $\overset{s}{\underset{\alpha =1}{\sum
}}\xi _{\alpha }$ is parallel to $T$. Then $\kappa _{1}=0$ or $\theta =\frac{%
\pi }{2},$ both of which is a contradiction. So, we have $\kappa _{2}\neq 0$
and $r\geq 3$. If we write equation (\ref{eq2}) in (\ref{eq1}), we get%
\begin{equation*}
\overset{s}{\underset{\alpha =1}{\sum }}\xi _{\alpha }=s\cos \theta T+\sqrt{s%
}\sqrt{1-s\cos ^{2}\theta }E_{3}.
\end{equation*}%
If we differentiate this last equation along the curve $\gamma ,$ we find%
\begin{equation}
\varphi T=\frac{-\kappa _{1}}{s\cos \theta }E_{2}-\frac{\kappa _{3}\sqrt{%
1-s\cos ^{2}\theta }}{\sqrt{s}}E_{4}.  \label{eq5}
\end{equation}%
If we calculate $g(\varphi T,\varphi T)$, we have
\begin{equation*}
s\cos \theta \left( 1-s\cos ^{2}\theta \right) \left( s\cos \theta -\kappa
_{3}^{2}\right) =\kappa _{1}^{2},
\end{equation*}%
which gives us $\kappa _{3}=$constant. In particular, if $\kappa _{3}=0$,
then we find equations (\ref{eq3}) and (\ref{eq4}). If $\kappa _{3}\neq 0,$
we differentiate equation (\ref{eq5}) along the curve $\gamma $ and find
that $\kappa _{4}=$constant. If we continue differentiating and calculating
the norm of both sides, we easily obtain $\kappa _{i}=$constant for all $i=%
\overline{1,r}$, that is, $\gamma $ is a slant helix of order $r$. Thus, we
have just proved the necessity.

To prove sufficiency, if $\gamma $ satisfies the equations given in the
Theorem, then it is easy to show that equation (\ref{cparalleltangent}) is
satisfied. So, $\gamma $ is $C$-parallel (in the tangent bundle).
\end{proof}

For $C$-parallel slant curves in the normal bundle, we have the following
Theorem:

\begin{theorem}
Let $\gamma :I\rightarrow M^{2m+s}$ be a unit-speed slant curve. Then $%
\gamma $ is $C$-parallel in the normal bundle if and only if it is a
Legendre helix of order $r\geq 3$ satisfying%
\begin{equation*}
\overset{s}{\underset{\alpha =1}{\sum }}\xi _{\alpha }=\sqrt{s}E_{3},
\end{equation*}%
\begin{equation*}
\varphi T=\frac{\kappa _{2}}{\sqrt{s}}E_{2}-\frac{\kappa _{3}}{\sqrt{s}}%
E_{4},
\end{equation*}%
\begin{equation*}
\kappa _{2}\neq 0,\text{ }\lambda =\frac{\kappa _{1}\kappa _{2}}{\sqrt{s}}
\end{equation*}%
and moreover if $\kappa _{3}=0$, then%
\begin{equation*}
\kappa _{2}=\sqrt{s},\text{ }\varphi T=E_{2}.\text{ }
\end{equation*}
\end{theorem}

\begin{proof}
Let us assume that $\gamma $ is $C$-parallel in the normal bundle. Then, if
we apply $T$ to equation (\ref{cparallelnormal}), we have $\eta ^{\alpha
}(T)=0$, so $\gamma $ is Legendre. Next, we apply $E_{2}$ and find $\kappa
_{1}=$constant. Thus, equation (\ref{cparallelnormal}) becomes%
\begin{equation*}
\kappa _{1}\kappa _{2}E_{3}=\lambda \overset{s}{\underset{\alpha =1}{\sum }}%
\xi _{\alpha },
\end{equation*}%
which gives us%
\begin{equation}
E_{3}=\frac{1}{\sqrt{s}}\overset{s}{\underset{\alpha =1}{\sum }}\xi _{\alpha
},\text{ }  \label{eq6}
\end{equation}%
\begin{equation*}
\kappa _{2}\neq 0,\text{ }\lambda =\frac{\kappa _{1}\kappa _{2}}{\sqrt{s}}.
\end{equation*}%
If we differentiate equation (\ref{eq6}), we get%
\begin{equation}
\varphi T=\frac{\kappa _{2}}{\sqrt{s}}E_{2}-\frac{\kappa _{3}}{\sqrt{s}}%
E_{4}.  \label{eq7}
\end{equation}%
If we differentiate this last equation, we obtain%
\begin{eqnarray}
\nabla _{T}\varphi T &=&\overset{s}{\underset{\alpha =1}{\sum }}\xi _{\alpha
}+\kappa _{1}\varphi E_{2}  \label{eq9} \\
&=&\frac{\kappa _{2}^{\prime }}{\sqrt{s}}E_{2}+\frac{\kappa _{2}}{\sqrt{s}}%
(-\kappa _{1}T+\kappa _{2}E_{3})-\frac{\kappa _{3}^{\prime }}{\sqrt{s}}E_{4}-%
\frac{\kappa _{3}}{\sqrt{s}}(-\kappa _{3}E_{3}+\kappa _{4}E_{5}).  \notag
\end{eqnarray}%
If we apply $E_{2}$ to both sides, we find $\kappa _{2}=$constant. Then, the
norm of equation (\ref{eq7}) gives us $\kappa _{3}=$constant. In particular,
if $\kappa _{3}=0$, from equation (\ref{eq7}), we have%
\begin{equation*}
\kappa _{2}=\sqrt{s},\text{ }\varphi T=E_{2}.
\end{equation*}%
Otherwise, from the norm of both sides in (\ref{eq9}), we also have $\kappa
_{4}=$constant. If we continue differentiating equation (\ref{eq9}), we find
that $\gamma $ is a helix of order $r$.

Conversely, let $\gamma $ be a Legendre helix of order $r\geq 3$ satisfying
the stated equations. Then, it is easy to show that equation (\ref%
{cparallelnormal}) is verified. Thus, $\gamma $ is $C$-parallel in the
normal bundle.
\end{proof}

\section{$C$-proper Slant Curves of $S$-manifolds\label{sect-cproper}}

\bigskip For $C$-proper slant curves in the tangent bundle, we can state the
following Theorem:

\begin{theorem}
\label{cproptan}Let $\gamma :I\rightarrow M^{2m+s}$ be a unit-speed slant
curve. Then $\gamma $ is $C$-proper (in the tangent bundle) if and only if
it is a non-Legendre slant curve satisfying%
\begin{equation*}
\overset{s}{\underset{\alpha =1}{\sum }}\xi _{\alpha }\in sp\left\{
T,E_{3},E_{4}\right\} ,
\end{equation*}%
\begin{equation*}
\varphi T\in sp\left\{ E_{2},E_{3},E_{4},E_{5}\right\} ,
\end{equation*}%
\begin{equation*}
\kappa _{1}\neq constant,\text{ }\kappa _{2}\neq 0,
\end{equation*}%
\begin{equation}
\text{ }\lambda =\frac{3\kappa _{1}\kappa _{1}^{\prime }}{s\cos \theta },
\label{eq10}
\end{equation}%
\begin{equation}
\kappa _{1}^{2}+\kappa _{2}^{2}=\frac{\kappa _{1}^{\prime \prime }}{\kappa
_{1}},  \label{eq11}
\end{equation}%
\begin{equation}
\lambda s\eta ^{\alpha }(E_{3})=-(2\kappa _{1}^{\prime }\kappa _{2}+\kappa
_{1}\kappa _{2}^{\prime }),  \label{eq12}
\end{equation}%
\begin{equation}
\lambda s\eta ^{\alpha }(E_{4})=-\kappa _{1}\kappa _{2}\kappa _{3},
\label{eq13}
\end{equation}%
\begin{equation}
\eta ^{\alpha }(E_{3})^{2}+\eta ^{\alpha }(E_{4})^{2}=\frac{1-s\cos
^{2}\theta }{s}  \label{eq14}
\end{equation}%
and moreover if $\kappa _{3}=0$, then%
\begin{equation}
\varphi T=\sqrt{1-s\cos ^{2}\theta }E_{2},  \label{eq16}
\end{equation}%
\begin{equation}
E_{3}=\frac{1}{\sqrt{s}\sqrt{1-s\cos ^{2}\theta }}\left( -s\cos \theta T+%
\overset{s}{\underset{\alpha =1}{\sum }}\xi _{\alpha }\right) ,  \label{eq17}
\end{equation}%
\begin{equation}
\kappa _{2}=\sqrt{s}\left( 1+\frac{\kappa _{1}\cos \theta }{\sqrt{1-s\cos
^{2}\theta }}\right) .  \label{eq18}
\end{equation}
\end{theorem}

\begin{proof}
Let $\gamma $ be $C$-proper (in the tangent bundle). If we apply $T$ to
equation (\ref{cpropertangent}), we find%
\begin{equation*}
\lambda s\cos \theta =3\kappa _{1}\kappa _{1}^{\prime }.
\end{equation*}%
Let us assume that $\gamma $ is Legendre. Then we have $\kappa _{1}^{\prime
}=0$, that is, $\kappa _{1}=$constant. if we apply $E_{2}$ to equation (\ref%
{cpropertangent}), we get%
\begin{equation*}
0=\kappa _{1}^{3}+\kappa _{1}\kappa _{2}^{2}-\kappa _{1}^{\prime \prime
}=\kappa _{1}\left( \kappa _{1}^{2}+\kappa _{2}^{2}\right)
\end{equation*}%
which gives us $\kappa _{1}=0$. Then equation (\ref{cpropertangent}) becomes%
\begin{equation*}
\lambda \overset{s}{\underset{\alpha =1}{\sum }}\xi _{\alpha }=0,
\end{equation*}%
which is a contradiction. Thus, $\gamma $ is non-Legendre slant and $\kappa
_{1}\neq $constant. We find equations (\ref{eq10}), (\ref{eq11}), (\ref{eq12}%
) and (\ref{eq13}) applying $T$, $E_{2}$, $E_{3}$ and $E_{4}$, respectively.
Then, we write these equations in (\ref{cpropertangent}) and calculate the
norm of boths sides to obtain equation (\ref{eq14}). Now, let us assume $%
\kappa _{2}=0$. Then, from equation (\ref{cpropertangent}), we have%
\begin{equation*}
\lambda \overset{s}{\underset{\alpha =1}{\sum }}\xi _{\alpha }=3\kappa
_{1}\kappa _{1}^{\prime }T,
\end{equation*}%
which is only possible when%
\begin{equation*}
T=\frac{1}{\sqrt{s}}\overset{s}{\underset{\alpha =1}{\sum }}\xi _{\alpha }.
\end{equation*}%
If we calculate $\nabla _{T}T$, we find $\kappa _{1}=0$, which is a
contradiction. Hence, $\kappa _{2}\neq 0$. Differentiating equation (\ref%
{cpropertangent}), we can easily see that%
\begin{equation*}
\varphi T\in sp\left\{ E_{2},E_{3},E_{4},E_{5}\right\} .
\end{equation*}%
In particular, if $\kappa _{3}=0$, we obtain equations (\ref{eq16}), (\ref%
{eq17}) and (\ref{eq18}). Please see our paper \cite{GO2018}, Case III,
equation (4.9), which is also valid when $\kappa _{1}$ and $\kappa _{2}$ are
not constants.

Conversely, if $\gamma $ is a non-Legendre slant curve satisfying the stated
equations, then Proposition \ref{lemma1} iii) is valid. So, $\gamma $ is
C-proper (in the tangent bundle).
\end{proof}

Finally, we give the following Theorem for $C$-proper slant curves in the
normal bundle:

\begin{theorem}
Let $\gamma :I\rightarrow M^{2m+s}$ be a unit-speed slant curve. Then $%
\gamma $ is $C$-proper in the normal bundle if and only if it is a Legendre
curve satisfying%
\begin{equation*}
\overset{s}{\underset{\alpha =1}{\sum }}\xi _{\alpha }\in sp\left\{
E_{3},E_{4}\right\} ,
\end{equation*}%
\begin{equation*}
\varphi T\in sp\left\{ E_{2},E_{3},E_{4},E_{5}\right\}
\end{equation*}%
\begin{equation*}
\kappa _{1}\neq constant,\text{ }\kappa _{2}\neq 0,
\end{equation*}%
\begin{equation*}
\kappa _{1}\kappa _{2}^{2}-\kappa _{1}^{\prime \prime }=0,
\end{equation*}%
\begin{equation*}
\lambda s\eta ^{\alpha }(E_{3})=-(2\kappa _{1}^{\prime }\kappa _{2}+\kappa
_{1}\kappa _{2}^{\prime }),
\end{equation*}%
\begin{equation*}
\lambda s\eta ^{\alpha }(E_{4})=-\kappa _{1}\kappa _{2}\kappa _{3},
\end{equation*}%
\begin{equation*}
\eta ^{\alpha }(E_{3})^{2}+\eta ^{\alpha }(E_{4})^{2}=\frac{1}{s}
\end{equation*}%
and moreover if $\kappa _{3}=0$, then%
\begin{equation*}
\overset{s}{\underset{\alpha =1}{\sum }}\xi _{\alpha }=\sqrt{s}E_{3},
\end{equation*}%
\begin{equation*}
\kappa _{2}=\sqrt{s},\text{ }\varphi T=E_{2}.\text{ }
\end{equation*}
\end{theorem}

\begin{proof}
The proof is similar to the proof of Theorem \ref{cproptan}. For the case $%
\kappa _{3}=0$, please refer to \cite{OG2014}.
\end{proof}

\section{Examples}

In this section, we give the following two examples in the well-known $S$%
-manifold $%
\mathbb{R}
^{2m+s}(-3s)$ \cite{Hasegawa}:

\begin{example}
Let us consider $%
\mathbb{R}
^{2m+s}(-3s)$ with $m=2$ and $s=2$. The curve $\gamma :I\rightarrow
\mathbb{R}
^{6}(-6)$ given by
\begin{equation*}
\gamma (t)=\left( \sin t,2+\sin t,-\cos t,3-\cos t,-2t-\sin t\cos
t,1-2t-\sin t\cos t\right)
\end{equation*}%
is a unit-speed non-Legendre slant helix with%
\begin{equation*}
\kappa _{1}=\kappa _{2}=\frac{1}{\sqrt{2}},\text{ }\theta =\frac{2\pi }{3}.
\end{equation*}%
It has the Frenet frame field%
\begin{equation*}
\left\{ T,\sqrt{2}\varphi T,\left( T+\overset{2}{\underset{\alpha =1}{\sum }}%
\xi _{\alpha }\right) \right\}
\end{equation*}
and it is C-parallel (in the tangent bundle) with $\lambda =\frac{1}{2}$.
\end{example}

\begin{example}
Let us consider $%
\mathbb{R}
^{2m+s}(-3s)$ with $m=1$ and $s=4$. We define real valued functions on an
open interval $I$ as
\begin{equation*}
\gamma _{1}(t)=2\int\limits_{0}^{t}\cos (e^{2u})du,\text{ }\gamma
_{2}(t)=-2\int\limits_{0}^{t}\sin (e^{2u})du,
\end{equation*}%
\begin{equation*}
\gamma _{3}(t)=...=\gamma _{6}(t)=-4\int\limits_{0}^{t}\cos (e^{2u})\left(
\int\limits_{0}^{u}\sin (e^{2v})dv\right) du.
\end{equation*}%
The curve $\gamma :I\rightarrow
\mathbb{R}
^{6}(-12)$, $\gamma (t)=\left( \gamma _{1}(t),...,\gamma _{6}(t)\right) $ is
a unit-speed Legendre curve with%
\begin{equation*}
\kappa _{1}=2e^{2t},\text{ }\kappa _{2}=2,\text{ }r=3,
\end{equation*}%
\begin{equation*}
\varphi T=E_{2},\text{ }E_{3}=\frac{1}{2}\overset{4}{\underset{\alpha =1}{%
\sum }}\xi _{\alpha }
\end{equation*}%
and it is $C$-proper in the normal bundle with $\lambda =-8e^{2t}$.
\end{example}

\textbf{Acknowledgements.} This work is financially supported by Balikesir University
Research Project Grant no. BAP 2018/016. The authors would like to thank the
Balikesir University.

\end{document}